\documentclass[12pt]{amsart}
\usepackage{graphicx}
\usepackage[centertags]{amsmath}
\usepackage{amsfonts}
\usepackage{amssymb}
\usepackage{amsthm}
\usepackage{newlfont}
\newtheorem{thm}{Theorem}[section]

\theoremstyle{definition}

\theoremstyle{remark}

\newcommand{\grad}{\textrm{grad}}

\title{On the geometry of four dimensional Riemannian manifold with a circulant metric and a circulant affinor structure}
\bigskip
\author{Dimitar Razpopov}
\date{}
\begin{document}

\maketitle

\begin{abstract}
We consider a $4$-dimensional Riemannian manifold $M$ with a
metric $g$ and an affinor structure $q$. We note the local
coordinates of $g$ and $q$ are circulant matrices. Their first
orders are $(A, B, C, B)$, $A, B, C \in FM$ and $(0, 1, 0, 0)$,
respectively.

Let $\nabla$ be the connection of $g$. Further, let $\mu_{1},
\mu_{2},\mu_{3}, \mu_{4}, \mu_{5}, \mu_{6}$ be the sectional
curvatures of $2$-sections $\{x, qx\}$, $\{x, q^{2}x\}$,
$\{q^{3}x, x\}$, $\{qx, q^{2}x\}$, $\{qx, q^{3}x\}$, $\{q^{2}x,
q^{3}x\}$ for arbitrary vector $x\in T_{p}M$, $p\in M$ . Then we
have that $q^{4}=E;$\quad $g(qx, qy)=g(x,y)$,\quad $x,\ y\in \chi
M$.

The main results of the present paper are

 1) \quad There exist a $q$-base in $T_{p}M$, $p\in M$.

 2) \quad if
$\nabla q=0$, then $\mu_{1}= \mu_{3}=\mu_{4}=\mu_{6}$, $\mu_{2}=
\mu_{5}=0$,

\end{abstract}

\Small{\textbf{Mathematics Subject Classification (2010)}: 53C15,
53B20}

\Small{\textbf{Keywords}: Riemannian metric, affinor structure,
sectional curvatures}
\thispagestyle{empty}
\section{Introduction}
The main purpose of the present paper is to continue the
considerations on the differential geometry of Riemannian
manifolds with a circulant metric $g$, and a circulant affinor
structure $q$ \cite{1}, \cite{2}, \cite{3}, \cite{4}.

Let $M$ be a Riemannian manifold with a metric $g$ and affinor
structure $q$, such that \cite{4}:
\begin{equation}\label{f1}
    g_{ij}=\begin{pmatrix}
      A & B & C & B \\
      B & A & B & C\\
      C & B & A & B\\
      B & C & B & A\\
    \end{pmatrix},\quad
    q_{i}^{.j}=\begin{pmatrix}
      0 & 1 & 0 & 0\\
      0 & 0 & 1 & 0\\
      0 & 0 & 0 & 1\\
      1 & 0 & 0 & 0\\
    \end{pmatrix},
\end{equation}
 where $det g_{ij}=(A-C)^{2}((A+C)^{2}-4B^{2})$, and $A=A(p),\ B=B(p),
\ C=C(p),\ p=(x^{1}, x^{2}, x^{3}, x^{4}) \in D \subset R^{4}$ are
 smooth functions in the local coordinate system $\{x^{1}, x^{2},
 x^{3}, x^{4}\}$. We suppose
 \begin{equation}\label{f2}
    0 < B < C < A.
\end{equation}
By virtue of \cite{2} we easily get that $g$ is positively defined
metric.

We know from \cite{4} that
\begin{equation}\label{f3}
    q^{4}=id, \qquad q\neq \pm id, \quad q^{2}\neq \pm id
\end{equation}
\begin{equation}\label{f4}
    g(qx, qy)=g(x, y),\qquad x, \ y \in \chi M.
\end{equation}
Let $\nabla$ be the Levi-Chivita connection of $g$ and $R$ be the
curvature of $\nabla$, i.e $R(x,
y)z=\nabla_{x}\nabla_{y}z-\nabla_{[x,y]}z$. We consider the
associated with $R$ tensor field $R$ of type $(0, 4)$, defined by
the condition
\begin{equation*}
    R(x, y, z, u)=g(R(x, y)z,u), \qquad x, y, z, u\in \chi M.
    \end{equation*}
Then we have
\begin{equation}\label{f5}
    \nabla q=0 \Leftrightarrow \grad A=\grad C.q^{2}, \grad B=\grad
    C.(q+q^{3})
\end{equation}
\begin{equation}\label{f6}
    \nabla q=0\Rightarrow R(x, y, qz, qu)= R(x, y, z, u), x, y, z,
    u\in\chi M.
\end{equation}
Further in our considerations we suppose that $M$ owns all
properties (\ref{f1})--(\ref{f6}).

\section{$q$-bases}
\begin{thm}
Let $M$ be a $4$-dimensional Riemannian manifold with a metric $g$
and an affinor structure $q$ with local coordinates (\ref{f1}). If
$p$ is a point in $M$ and $x$ is an arbitrary vector in $T_{p}M$
then there exist a $q$-base in $T_{p}M$.
\end{thm}
\begin{proof}
Let $p\in M$, $x=(x^{1}, x^{2}, x^{3}, x^{4})\in T_{p}M$, and
$qx\neq x$, $q^{2}x\neq x$. Then using (\ref{f1}) we easily see
that the determinant of the coordinates of $x, qx, q^{2}x$ and
$q^{3}x$ is not zero. So, $\{x, qx, q^{2}x, q^{3}x\}$ is a
$q$-base in $T_{p}M$.
\end{proof}
\begin{thm}
Let $M$ be a $4$-dimensional Riemannian manifold with a metric $g$
and an affinor structure $q$ with local coordinates (\ref{f1}). If
$p$ is a point in $M$ and $x$ is an arbitrary vector in $T_{p}M$
then there exist an orthonormal $q$-base in $T_{p}M$.
\end{thm}
\begin{proof}
Let $p\in M$, $x=(x^{1}, x^{2}, x^{3}, x^{4})\in T_{p}M$. We
suppose that $$x^{4}=0,\qquad
x^{2}=\dfrac{\sqrt{A+B-2C}-\sqrt{A+B+2C}}{2\sqrt{A+B-2C}\sqrt{A+B+2C}},$$
and $x^{1}$ and $x^{3}$ satisfy the following system
\begin{align}\label{f7}\nonumber
    x^{1}+x^{3}&=\dfrac{\sqrt{A+B-2C}-\sqrt{A+B+2C}}{2\sqrt{A+B-2C}\sqrt{A+B+2C}},
    \\
    x^{1}x^{3}&=\dfrac{2B^{2}-C^{2}-AC}{2(A-C)\sqrt{A+B-2C}\sqrt{A+B+2C}}.
\end{align}
The solutions of the system (\ref{f7}) are solutions of the square
equations with discriminant
$$D=(x^{1}+x^{3})^{2}-4x^{1}x^{3}=\dfrac{A^{2}+C^{2}+2AC+(A-C)\sqrt{(A+C)^{2}-4B^{2}}-4B^{2}}{2{A-C}((A+C)^{2}-C)-4B^{2}}.$$
We have $D > 0$ because of (\ref{f2}).

Now we can verify directly that $g(x, x)=1$, $g(x, qx)=0$, $g(x,
q^{2}x)=0$ and consequently, on due to (\ref{f1}), we see that
$g(x, q^{3}x)=g(qx, q^{2}x)=g(qx, q^{3}x)=g(q^{2}x, q^{3}x)=0$.
So, $\{x, qx, q^{2}x, q^{3}x\}$ is an orthonormal $q$-base.
\end{proof}
\textbf{An example.} Let $p\in M$ and $\{x, qx, q^{2}x, q^{3}x\}$
be a $q$-base in $T_{p}M$. Let $L,\ N,\ S,\ T$ be four points in
$T_{p}M$ such that $PL=x$, $PN=qx$, $PS=q^{2}x$, $PT=q^{3}x$. Let
$\alpha$ and $\beta$ be the corresponding angels between $x$ and
$qx$, and $x$ and $q^{2}x$, respectively. Using
(\ref{f1})--(\ref{f4}) we get $
\begin{Vmatrix}
  LN \\
\end{Vmatrix}=\begin{Vmatrix}
  LT \\
\end{Vmatrix}=\begin{Vmatrix}
  ST \\
\end{Vmatrix}=2\begin{Vmatrix}
  x \\
\end{Vmatrix}^{2}(1-\cos \alpha)$,

$
\begin{Vmatrix}
  LS \\
\end{Vmatrix}=\begin{Vmatrix}
  NT \\
\end{Vmatrix}=2\begin{Vmatrix}
  x \\
\end{Vmatrix}^{2}(1-\cos \beta)$. We see that $LNST$ is a pyramid
with six sides which are all congruent isoseles triangles. The
angles of such a side of the pyramid we note by $\gamma$ and
$\delta$, so $2\gamma +\delta=\pi$. They satisfy
\begin{equation*}
    \cos \gamma =\dfrac{1-\cos \beta}{2\sqrt{1-\cos \alpha}\sqrt{1-\cos \beta}},
    \quad
    cos \delta =\dfrac{1-2\cos \alpha+\cos \beta}{2(1-cos \alpha)}.
    \end{equation*}
    Evidently, in the case when $\{x, qx, q^{2}x, q^{3}x\}$ is an
    orthonormal $q$-base then the every side of the pyramid is an
    equilateral triangle.

\section{curvatures}

Let $p$ be a point in $M$ and $x$, $y$ be two linearly independent
vectors on $T_{p}M$. It is known the value
\begin{equation}\label{3.3}
    \mu(E;p)=\frac{R(x, y, x, y)}{g(x, x)g(y, y)-g^{2}(x, y)}
\end{equation}
is the sectional curvature of $2$-section $E=\{x, y\}$.

For a vector $x=(x^{1}, x^{2}, x^{3}, x^{4})$ from $T_{p}M$ we
suppose
$((x^{1}-x^{3})^{2}+(x^{2}-x^{4})^{2})(x^{1}-x^{2}+x^{3}-x^{4})(x^{1}+x^{2}+x^{3}+x^{4})\neq
0.$ By using (\ref{f3}) we get that the vectors $x, qx, q^{2}x,
q^{3}x$ are linearly independent.
\begin{thm}
Let $M$ be the Riemannian manifold with a metric $g$ and a
parallel structure $q$, defined by (\ref{f1}) and (\ref{f2}),
respectively. Let $p$ be a point in $M$ and $x$ be an arbitrary
vector in $T_{p}M$. Then the sectional curvatures of $2$-sections
$E_{1}=\{x, qx\}$, $E_{3}=\{q^{3}x, x\}$, $E_{4}=\{q^{2}x, qx\}$,
$E_{6}=\{q^{2}x, q^{3}x\}$ are equal among them, and the sectional
curvatures of $2$-sections $E_{2}=\{x, q^{2}x\}$ $E_{5}=\{q^{3}x,
qx\}$ are vanish.
\end{thm}
\begin{proof}
From (\ref{f6}) we find
\begin{equation}\label{3.6}
    R(x, y, z, u)=R(x, y, qz, qu)=R(x, y, q^{2}z, q^{2}u)=R(x, y, q^{3}z, q^{3}u).
    \end{equation}
    In (\ref{3.6}) we get the following substitutions:

    a) $z=x$, $y=u=qx$; $x\sim qx$, $z=x$, $y=u=q^{2}x$; $z=x$,
    $y=u=q^{3}x$.
    \begin{equation}\label{dop1}
    R(x, qx, x, qx)=R(x, q^{3}x, x, q^{3}x),\qquad R(x, q^{2}x, x,
    q^{2}x)=0.
    \end{equation}

 b) $z=x$, $y=qx$, $u=q^{2}x$; $z=qx$, $y=u=q^{2}x$; $z=x$,
    $y=q^{3}x$, $u=q^{2}x$; $z=qx$, $y=u=q^{3}x$.
    \begin{equation*}
           R(x, qx, x, q^{2}x)=R(x, q^{2}x, qx, q^{2}x)= R(x, q^{3}x, q^{2}x, x)=R(x, q^{3}x, qx, q^{3}x)=0
    \end{equation*}

c) $z=x$, $y=qx$, $u=q^{3}x$; $z=qx$, $y=u=q^{2}x$;  $z=qx$,
    $y=qx$, $u=q^{2}x$; $z=q^{2}x$, $y=qx$, $u=q^{3}x$.
    \begin{equation*}
           -R(x, qx, x, q^{3}x)=R(x, qx, qx, q^{2}x)= R(x, qx, q^{2}x, q^{3}x)=R(x, qx, x, qx)
    \end{equation*}

    d) $z=qx$, $y=qx$, $u=q^{3}x$; $z=qx$, $y=q^{2}x$, $u=q^{3}x$;  $z=q^{2}x$,
    $y=q^{2}x$, $u=q^{3}x$; $z=q^{2}x$, $y=qx$, $u=q^{3}x$.
    \begin{equation*}
           R(x, q^{2}x, qx, q^{3}x)=R(x, q^{2}x, q^{2}x, q^{3}x)= R(qx, q^{2}x, qx,
           q^{3}x)=0
    \end{equation*}

    e) $z=qx$, $y=q^{3}x$, $u=q^{2}x$; $z=q^{2}x$, $y=u=q^{3}x$;  $x\sim qx$, $z=q^{2}x$,
    $y=q^{2}x$, $u=q^{3}x$; $x\sim qx$, $z=q^{2}x$, $y=u=q^{3}x$.
    \begin{equation*}
           -R(x, q^{3}x, qx, q^{2}x)=-R(x, q^{3}x, q^{2}x, q^{3}x)=R(qx, q^{2}x, q^{2}x, q^{3}x)= R(x, qx, x,
           qx)
    \end{equation*}
    and
    \begin{equation*}
           R(qx, q^{3}x, q^{2}x, q^{3}x)=0
    \end{equation*}

    f) $x\sim qx$, $z=qx$, $y=q^{3}x$, $u=q^{3}x$; $x\sim qx$, $z=qx$, $y=q^{2}x$, $u=q^{2}x$; $x\sim q^{2}x$, $z=q^{2}x$,
    $y=q^{3}x$, $u=q^{3}x$.
    \begin{equation}\label{dop7}
           R(qx, q^{3}x, qx, q^{3}x)=0
    \end{equation}
     \begin{equation}\label{dop8}
           R(qx, q^{2}x, qx, q^{2}x)=R(q^{2}x, q^{3}x, q^{2}x, q^{3}x)= R(x, qx, x,
           qx)
    \end{equation}
 From (\ref{f4}), (\ref{3.3}), (\ref{dop1})--(\ref{dop8}) we get
    \begin{equation*}
    \mu(E_{1};p)= \mu(E_{3};p)= \mu(E_{4};p)=\mu(E_{6};p)=\frac{R(x, qx, x, qx)}{g^{2}(x, x)-g^{2}(x, qx)}.
\end{equation*}
\begin{equation*}
    \mu(E_{2};p)= \mu(E_{5};p)= 0.
\end{equation*}
By virtue of the linear independence of $x$ and $qx$ we have
$$g^{2}(x,x)-g^{2}(x, qx)=g^{2}(x,x)(1-cos\varphi)\neq 0, $$  where
$\varphi$ is the angle between $x$ and $qx$.
\end{proof}

\author{Dimitar Razpopov\\ Department of Mathematics and Physics\\ Agricultural University of Plovdiv\\
Bulgaria 4000\\
e-mail:drazpopov@qustyle.bg}

\begin{thebibliography}{D}
\bibitem{1}
\textsc{G. Dzhelepov, I. Dokuzova, D. Razpopov}. \emph{On a three
dimensional Riemannian manifold with an additional structure}.
arXiv:math.DG/0905.0801
\bibitem{2}
\textsc{G. Dzhelepov, D. Razpopov, I. Dokuzova}. \emph{Almost
conformal transformation on Riemannian manifold with an additional
structure}, Proceedings of the Anniversary International
Conference - REMIA, Plovdiv, (2010), 125--128,
arXiv:math.DG/1010.4975
\bibitem{3}
\textsc{I. Dokuzova, D. Razpopov}. \emph{On affine connections in
a Riemannian manifold with a circulant metric and two circulant
affinor structure}, Proceedings of the 40-th Anniversary Spring
Conference of UBM, Borovets, (2011), 170--176
\bibitem{4}
\textsc{K. Yano}. \emph{Differential geometry}. Pergamont press,
New York, 1965
\bibitem{5}
\textsc{H.Hristov}. \emph{Mathematical methods in physics}.
Science and Art, Sofia, 1967(in bulgarian)
\end{thebibliography}
\end{document}